\documentclass[11pt]{amsart}
\usepackage{amsfonts,amsmath,amsthm,amscd,amssymb,latexsym,cite,verbatim,texdraw,floatflt,
caption2}
\RequirePackage{hyperref}
\usepackage[a4paper]{geometry}

\newcommand{\diam}{\mathrm{diam}\,}
\newcommand{\add}{\mathrm{add}}
\newcommand{\cof}{\mathrm{cof}}
\newcommand{\cov}{\mathrm{cov}}
\newcommand{\comb}{\operatornamewithlimits{\perp\!\perp}}

\newtheorem{theorem}{Theorem}
\newtheorem{example}{Example}
\newtheorem{question}{Question}

\title{Constructing balleans}
\author{Taras Banakh and   Igor  Protasov}
\address{T.Banakh: Ivan Franko National University of Lviv (Ukraine) and 
Jan Kochanowski University in Kielce (Poland)}
\email{t.o.banakh@gmail.com}
\address{I.Protasov: Faculty of Computer Science and Cybernetics, Kyiv University,          Academic Glushkov pr. 4d, 03680 Kyiv, Ukraine}
\email{i.v.protasov@gmail.com}

\begin{document}
\begin{abstract} A ballean is a set endowed with a coarse  structure. We introduce and explore three constructions of balleans  from a pregiven family of balleans:  bornological products, bouquets and combs. 
We analyze  the smallest and the largest coarse structures on a set $X$ compatible with a given bornology on $X$.
\end{abstract}
\maketitle

\section{ Introduction}
\normalsize

Given a set $X$, a family $\mathcal{E}$  of subsets of $X\times X$ is called a
{\it  coarse structure} on $X$ if
\begin{itemize}
\item{} each $E \in \mathcal{E}$  contains the diagonal $\bigtriangleup _{X}:=\{(x,x): x\in X\}$ of $X$;

\item{}  if  $E$, $E^{\prime} \in \mathcal{E}$  then  $E \circ E^{\prime} \in \mathcal{E}$  and
$ E^{-1} \in \mathcal{E}$,    where  $E \circ E^{\prime} = \{  (x,y): \exists z\;\; ((x,z) \in E,  \ (z, y)\in E^{\prime})\}$,    $ E^{-1} = \{ (y,x):  (x,y) \in E \}$;

\item{} if $E \in \mathcal{E}$ and  $\bigtriangleup_{X}\subseteq E^{\prime}\subseteq E$  then  $E^{\prime} \in \mathcal{E}$.
\end{itemize}
Elements $E\in\mathcal E$ of the coarse structure are called {\em entourages} on $X$.

For $x\in X$  and $E\in \mathcal{E}$ the set $E[x]:= \{ y \in X: (x,y)\in\mathcal{E}\}$ is called the {\it ball of radius  $E$  centered at $x$}.
Since $E=\bigcup_{x\in X}\{x\}\times E[x]$, the entourage $E$ is uniquely determined by  the family of balls $\{ E[x]: x\in X\}$.
A subfamily $\mathcal B\subset\mathcal E$ is called a {\em base} of the coarse structure $\mathcal E$ if each set $E\in\mathcal E$ is contained in some $B\in\mathcal B$.

The pair $(X, \mathcal{E})$  is called a {\it coarse space}  \cite{b1} or  a {\em ballean} \cite{b2}, \cite{b3}.
In \cite{b2} every base of a coarse structure, defined in terms of balls, is called a {\it ball structure}.
We prefer the name balleans not only by the author’s rights but also because a  coarse spaces sounds like some special type of topological spaces.
In fact,   balleans can be considered as non-topological  antipodes of uniform topological spaces.
Our compromise with \cite{b1} is in usage the name coarse structure in place of the ball structure.

In this paper, all balleans under consideration are supposed to be
 {\it connected}: for any $x, y \in X$, there is $E\in \mathcal{E}$ such $y\in E[x]$.
A subset  $Y\subseteq  X$  is called {\it bounded} if $Y= E[x]$ for some $E\in \mathcal{E}$,
  and $x\in X$.
The family $\mathcal{B}_{X}$ of all bounded subsets of $X$  is a bornology on $X$.
We recall that a family $\mathcal{B}$  of subsets of a set $X$ is a {\it bornology}
if $\mathcal{B}$ contains the family $[X] ^{<\omega} $  of all finite subsets of $X$
 and $\mathcal{B}$  is closed   under finite unions and subsets. A bornology $\mathcal B$ on a set $X$ is called {\em unbounded} if $X\notin\mathcal B$.

Each subset $Y\subseteq X$ defines a {\it subbalean}  $(Y, \mathcal{E}|_{Y})$  of $(X, \mathcal{E})$,
 where $\mathcal{E}|_{Y}= \{ E \cap (Y\times Y): E \in \mathcal{E}\}$.
A  subbalean $(Y, \mathcal{E}|_{Y})$  is called  {\it large} if there exists $E\in \mathcal{E}$
 such that $X= E[Y]$, where $E[Y]=\bigcup _{y\in Y} E[y]$.

Let $(X, \mathcal{E})$, $(X^{\prime}, \mathcal{E}^{\prime})$
 be  balleans. A mapping $f: X \to X^{\prime}$ is called
  {\it coarse (or macrouniform) }  if, for every $E\in \mathcal{E}$, there
  exists $E^{\prime}\in \mathcal{E}$  such that $f(E(x))\subseteq  E^{\prime}(f(x))$
    for each $x\in X$.
If $f$ is a bijection such that $f$  and $f ^{-1 }$ are coarse then   $f  $  is called an {\it asymorphism}.
If  $(X, \mathcal{E})$ and  $(X^{\prime}, \mathcal{E}^{\prime})$  contains large  asymorphic  subballeans then they are called {\it coarsely equivalent.}

For coarse spaces  $(X_{\alpha}, \mathcal{E}_{\alpha})$, $\alpha < \kappa$, their product is the Cartesian product $X=\prod_{\alpha<\alpha}X_\alpha$ endowed with the coarse structure generated by the base consisting of the entourages $$\big\{\big((x_\alpha)_{\alpha<\kappa},(y_\alpha)_{\alpha<\kappa}\big)\in X\times X:\forall \alpha<\kappa\;\;(x_\alpha,y_\alpha)\in E_\alpha\big\},$$ where $(E_\alpha)_{\alpha<\kappa}\in\prod_{\alpha<\kappa}\mathcal E_\alpha$.

A class $\mathfrak{M}$ of balleans is called a {\em variety} if $\mathfrak{M}$
 is closed under formation of subballeans, coarse images and Cartesian products. For characterization    of all varieties of balleans, see \cite{b4}.

Given a family $\mathfrak{F}$ of subsets of $X\times X$, we denote by $\mathcal{E}$ the intersection of all  coarse  structures,  containing  each $F\cup \bigtriangleup _{X}$,  $F\in \mathfrak{F}$,   and say that  $\mathcal{E}$ is generated by $\mathfrak{F}$.
It is easy to see that $\mathcal{E}$ has a base of subsets of the form  $E _1 \circ E _{1}\circ  \ldots\circ E _{n}$ , where
$$E_1,\dots,E_n\in\{F\cup F^{-1}\cup \{(x,y)\}\cup\bigtriangleup_{X}: F\in\mathfrak{F},\;\;x,y\in X\}.$$

By a {\em pointed ballean} we shall understand a ballean $(X,\mathcal E)$ with a  distinguished point $e_*\in X$.

\section{ Metrizability and normality}

Every  metric $d$ on a set  $X$ defines the coarse  structure
$\mathcal{E} _{d}$  on $X$  with  the base $\{\{(x,y): d(x, y)< n\} : n\in \mathbb{N}\}$.
A ballean $(X, \mathcal{E})$ is called {\it metrizable}  if there is a metric $d$ on such that $\mathcal{E} = \mathcal{E} _{d}$.

\begin{theorem}[\cite{b5}]\label{t1}
A ballean $(X,\mathcal{E})$ is metrizable  if and only if $\mathcal{E}$  has a countable base.
\end{theorem}

Let $(X,\mathcal{E})$  be a ballean. A subset $U \subseteq X$  is called an
 {\it asymptotic neighbourhood  } of a subset $Y\subseteq X$ if, for every $E\in \mathcal{E}$ the set $E [Y] \setminus U$  is bounded.

Two subset $Y, Z$ of $X$  are called  {\it asymptotically disjoint  (separated)}  if, for every $E\in \mathcal{E}$,
  $E[Y] \cap E[Z] $  is bounded ($Y$ and $Z$ have disjoint asymptotic neighbourhoods).

A ballean $(X, \mathcal{E})$  is called {\it normal}  \cite{b6} if any
two asymptotically disjoint subsets of $X$  are asymptotically separated. Every ballean $(X, \mathcal{E})$
 with linearly ordered base of $\mathcal{E}$  is  normal.
In particular,   every  metrizable ballean is normal, see \cite{b6}.

A function $f: X\to \mathbb{R}$  is called   {\it slowly oscillating} if 
 for any $E\in \mathcal{E}$  and $\varepsilon > 0$,  there exists a bounded  subset $B$ of $X$  such that $\diam   f (E[x])<  \varepsilon $ for each $x\in X \setminus B$.

\begin{theorem}[\cite{b6}]\label{t2}
A ballean  $(X, \mathcal{E})$ is normal if and only if,
 for any two disjoint and asymptotically disjoint
subsets $Y, Z$ of $X$, there exists a slowly oscillating function  $f: X\to [0,1]$ such that $f(Y)\subset\{0\}$  and $f(Z)\subset\{1\}$.
\end{theorem}

For any unbounded bornology $\mathcal{B}$ on  a set $X$ the cardinals
$$
\begin{aligned}
&\add(\mathcal B)=\min\{\mathcal A\subset\mathcal B:\bigcup\mathcal A\notin\mathcal B\},\\
&\cof(\mathcal B)=\min\{\mathcal C\subset\mathcal B:\forall B\in\mathcal B\;\;\exists C\in\mathcal C\;\;B\subset C\}\mbox{ and}\\
&\cov(\mathcal B)=\min\{|\mathcal C|:\mathcal C\subset\mathcal B,\;\bigcup\mathcal C=X\}
\end{aligned}
$$
are called the {\em additivity}, the {\em cofinality}, and the {\em covering number}  of $\mathcal B$, respectively. It is well-known (and easy to see) that $\add(\mathcal B)\le\cov(\mathcal B)\le\cof(\mathcal B)$.

The following  theorem was proved in \cite[1.4]{b3}.

\begin{theorem}\label{t3} If the product $X\times Y$  of balleans $X, Y$  is normal  then $$\add(\mathcal{B} _{X})=\cof(\mathcal B_X)=\cof(\mathcal B_Y)=\add(\mathcal B_Y).$$
\end{theorem}

 \begin{theorem}\label{t4} Let $X$ be the Cartesian product of a family  $\mathcal{F}$ of metrizable balleans.
Then the following statements are equivalent:
\begin{enumerate}
\item $X$  is metrizable;
\item $X$  is normal;
\item all but finitely many balleans from $\mathcal{F}$  are bounded.
\end{enumerate}
\end{theorem}

\begin{proof} We need only to show $(2)\Rightarrow  (3)$.
Assume the contrary.  Then there exists a family  $(Y_{n})_{n< \omega}$  of unbounded
metrizable balleans such that the Cartesian product $Y=\prod_{n\in\omega} Y_{n}$
  is normal.  On the other hand, $\add(\mathcal B_Y)\le\add(\mathcal B_{Y_0})=\aleph_0$ and a standard diagonal  argument shows that $\cof(\mathcal{B}_{Y})>\aleph_{0}$,    contradicting Theorem~\ref{t3}. 
\end{proof}

\section{Bornological products}

Let $\{(X_{\alpha}, \mathcal{E}_{\alpha}): \alpha\in A\}$
 be an indexed family of pointed balleans and let $\mathcal{B}$  be a bornology on the index set $A$. For each $\alpha\in A$ by $e_\alpha$ we denote the distinguished point of the ballean $X_\alpha$.

The {\em $\mathcal B$-product} of the family of pointed balleans $\{X_\alpha:\alpha\in A\}$ is the set 
$$X_{\mathcal{B}}=\big\{(x_{\alpha})_{\alpha\in A}\in \prod_{\alpha\in A} X_{\alpha}: \{\alpha\in A:x_{\alpha}\ne e_{\alpha}\}\in\mathcal B\big\},$$
 endowed with the coarse structure $\mathcal E_{\mathcal B}$, generated by the base consisting of the entourages 
$$\big\{\big((x_\alpha)_{\alpha\in A},(y_\alpha)_{\alpha\in A}\big)\in X_{\mathcal B}\times X_{\mathcal B}:\forall \alpha\in B\;\;(x_\alpha,y_\alpha)\in E_\alpha\big\}$$where $B\in\mathcal B$ and $(E_\alpha)_{\alpha\in B}\in\prod_{\alpha\in B}\mathcal E_\alpha$.


For the bornology  $\mathcal{B}=\mathcal{P}_{A}$ consisting of all subsets of the index set $A$, the $\mathcal{B}$-product $X_{\mathcal B}$ coincides with the Cartesian product
$\prod_{\alpha\in A}( X_a , \mathcal{E}_{a})$ of coarse spaces $(X_\alpha,\mathcal E_\alpha)$.

If each $X_{\alpha}$  is the doubleton $\{0,1\}$ with distnguished point  $e_{\alpha}=0$, then
the $\mathcal{B}$-product
  is called the {\it $\mathcal{B}$-macrocube } on $A$.
If $|A|= \omega$  and  $\mathcal{B}=[A]^{< \omega}$  then  we get the well-known  Cantor macrocube, whose coarse characterization was given by Banakh and Zarichnyi in \cite{b10}.

For relations between macrocubes and hyperballeans, see \cite{b8}, \cite{b9}.

\begin{theorem}\label{t5}
Let $\mathcal{B}$ be a bornology on a set  and let
$X_{\mathcal{B}}$ be the $\mathcal{B}$-product
 of a family  of unbounded metrizable pointed balleans.  Then the following statements are equivalent:
\begin{enumerate}
\item $X_{\mathcal{B}}$ is metrizable;
\item $X_{\mathcal{B}}$ is normal;
\item $|A|= \omega$ and $\mathcal{B} = [A]^{<\omega}$.
\end{enumerate}
\end{theorem}

\begin{proof}
To  see that  $(2) \Rightarrow (3)$,  repeat the proof of  Theorem \ref{t4}.
\end{proof}

\begin{theorem}\label{t6} Let  $\mathcal{B}$  be a bornology on a set $A$  and let
 $X_{\mathcal{B}} $ be the $\mathcal{B}$-product of a family $\{X_\alpha:\alpha\in A\}$ of bounded pointed balleans which are not singletons. The coarse space $X_{\mathcal{B}} $ is metrizable if and only if the bornology $\mathcal{B}$  has a countable base.
\end{theorem}

\begin{proof} Apply Theorem~\ref{t1}.
\end{proof}

Let $X$  be a macrocube  on a set $A$ and $Y$  be a macrocube on a set $B$,
$A \cap B = \emptyset$.
Then $X \times Y$ is a macrocube  on $X\cup Y$  and, by Theorem 3,
 $X\times Y$  needs not to be normal.

\begin{question}\label{q1}
How can one detect whether a given macrocube is normal?
Is a  $\mathcal{B}$-macrocube on an infinite set $A$ normal  provided that $\mathcal{B}\ne\mathcal P_A$  is a maximal unbounded bornology on $A$?
\end{question}

Let $\{X_{n}: n<\omega \}$
 be a family of finite balleans, $\mathcal{B}=[ \omega ]^{<\omega} $.
By [10], the  $\mathcal{B}$-product of the family
$\{X_{n}: n<\omega \}$ is coarsely  equivalent to the Cantor macrocube.
\vspace{5 mm}

\begin{question}\label{q2} Let $\{X_{\alpha}: \alpha\in A\}$  be a family of finite (bounded) pointed balleans and let $\mathcal{B}$ be a bornology on $A$.
How can one detect whether a $\mathcal{B}$-product  of $\{X_{\alpha}:\alpha\in A\}$
 is coarsely equivalent to some macrocube?
\end{question}

\section{Bouquets}

Let $\mathcal B$ be a bornology on a set  $A$   and  let
$\{(X _{\alpha}, \mathcal{E} _{\alpha}) : \alpha\in A\}$ be a   family of pointed balleans. The subballean
$$\bigvee_{\alpha\in A}X_\alpha:=\big\{(x_\alpha)_{\alpha\in A}\in X_{\mathcal B}:|\{\alpha\in A:x_\alpha\ne e_\alpha\}|\le 1\big\}$$of the $\mathcal B$-product $X_{\mathcal B}$ is called the {\em $\mathcal B$-bouquet} of the family  $\{(X_\alpha,\mathcal E_\alpha):\alpha\in A\}$.
The point $e=(e_\alpha)_{\alpha\in A}$ is the distinguished point of the ballean $\bigvee_{\alpha\in A}X_\alpha$.

For every $\alpha\in A$ we identify the ballean $X_\alpha$ with the subballean $\{(x_\beta)_{\beta\in A}\in X_{\mathcal B}:\forall \beta\in A\setminus \{\alpha\}\;\;x_\beta=e_\beta\}$ of $\bigvee_{\alpha\in A}X_\alpha$. Under such identification $\bigvee_{\alpha\in A}X_\alpha=\bigcup_{\alpha\in A}X_\beta$ and $X_\alpha\cap X_\beta=\{e\}=\{e_\alpha\}=\{e_\beta\}$ for any distinct indices $\alpha,\beta\in A$.

Applying Theorem 1, we get the following two  theorems.

\begin{theorem}\label{t7}
Let $\mathcal{B}$  be a bornology on a set $A$  and let
$\{X_{\alpha}: \alpha\in A\}$
 be a family of unbounded pointed metrizable balleans.
The $\mathcal{B}$-bouquet $\bigvee_{\alpha\in A}X_\alpha$ 
 is  metrizable if and only if $|A|= \omega$  and  $\mathcal{B} = |A|^{<\omega} $.
 \end{theorem}
 
\begin{theorem}\label{t8} 
Let $\mathcal{B}$  be a bornology on a set $A$  and let  $\{X_\alpha:\alpha\in A\}$ be a family of bounded pointed balleans, which are not singletons.
The $\mathcal{B}$-bouquet $\bigvee_{\alpha\in A}X_\alpha$ 
 is  metrizable if and only if the bornology $\mathcal{B}$ has a countable base.
\end{theorem}

\begin{theorem}\label{t9} A bornological bouquet of any family  of pointed normal balleans is normal.
\end{theorem}

\begin{proof} Let $\mathcal B$ be a bornology on a non-empty set $A$ and $X$ be the  $\mathcal{B}$-bouquet of pointed normal balleans  $X_{\alpha}$, $\alpha\in A$. Given two disjoint asymptotically disjoint sets $Y,Z\subset X$, we shall construct a slowly oscillating function $f:X\to[0,1]$ such that $f(Y)\subset\{0\}$ and $f(Z)\subset \{1\}$. 
The definition of the coarse structure on the $\mathcal B$-bouquet ensures that for every $\alpha\in A$ the subsets  $Y\cap X_\alpha$ and $Z\cap X_\alpha$ are asymptotically disjoint in the coarse space $X_\alpha$, which is identified with the subspace $\{(x_\beta)\in X:\forall \beta\in A\setminus\{\alpha\}\;\;x_\beta=e_\beta\}$ of the $\mathcal B$-bouquet $X$. By the normality of $X_\alpha$, there exists a slowly oscillating function $f_\alpha:X_\alpha\to[0,1]$ such that $f_\alpha(Y\cap X_\alpha)\subset\{0\}$ and $f_\alpha(Z\cap X_\alpha)\subset [0,1]$. Changing the value of $f_\alpha$ in the distinguished point $e_\alpha$ of $X_\alpha$, we can assume that $f_\alpha(e_\alpha)=f_\beta(e_\beta)$ for any $\alpha,\beta\in A$. Then the function $f:X\to[0,1]$, defined by $f{\restriction}X_\alpha=f_\alpha$ for $\alpha\in A$ is slowly ascillating and has the desired property: $f(Y)\subset \{0\}$ and $f(Z)\subset\{1\}$. By  Theorem~\ref{t2}, the ballean $X$ is normal.
 \end{proof}

\section{Combs}

Let $(X, \mathcal{E})$  be a ballean and $A$  be a subset of $X$.
Let $\{(X_{\alpha}, \mathcal{E}_{\alpha}): \alpha\in A\}$  be a family of pointed balleans with the marked points  $e_{\alpha}\in X_{\alpha}$ for $\alpha\in A$.

The bornology $\mathcal B_X$ of the ballean $(X,\mathcal E)$ induces a bornology $\mathcal B:=\{B\in\mathcal B_X:B\subset A\}$ on the set $A$. Let $\bigvee_{\alpha\in A}X_\alpha$ be the $\mathcal B$-bouquet of the family of pointed balleans $\{(X_\alpha,\mathcal E_\alpha):\alpha\in A\}$, and let $e$ we denote the distinguished point of the bouquet $\bigvee_{\alpha\in A} X_\alpha$.

For for every $\alpha\in A$ we identify the ballean $X_\alpha$ with the subballean $\{(x_\beta)_{\beta\in A}\in\bigvee_{\alpha\in A}X_\alpha:\forall \beta\in A\setminus\{\alpha\}\;\;x_\beta=e_\beta\}$ of $\bigvee_{\alpha\in A}X_\alpha$. Then $\bigvee_{\alpha\in A}X_\alpha=\bigcup_{\alpha\in A}X_\alpha$ and $X_\alpha\cap X_\beta=\{e\}=\{e_\alpha\}=\{e_\beta\}$ for any distinct indices $\alpha,\beta\in A$.

The suballean
$$X\comb\limits_{\alpha\in A}X_\alpha:=(X\times\{e\})\cup\bigcup_{\alpha\in A}(\{\alpha\}\times X_\alpha)$$ of the ballean $X\times \textstyle{\bigvee\limits_{\alpha\in A}}X_\alpha$  is called the {\em comb} with handle $X$ and spines $X_\alpha$, $\alpha\in A\subset X$. We shall identify the handle $X$ and the spines $X_\alpha$ with the subsets $X\times\{e\}$ and $\{\alpha\}\times X_\alpha$ in the comb $X\comb\limits_{\alpha\in A}X_\alpha$.


\begin{theorem}\label{t10} The comb $X\comb\limits_{\alpha\in A}X_\alpha$ is metrizable if the balleans $X$ and $X_\alpha$, $\alpha\in A$, are metrziable, and for each bounded set $B\subset X$ the intersection $A\cap B$ is finite.
\end{theorem}

\begin{proof} Applying Theorem~\ref{t7}, we conclude that the bouquet $\bigvee_{\alpha\in A}X_\alpha$ is metrizable. Then the comb $X\comb\limits_{\alpha\in A}X_\alpha$ is metrizable being a subspace of the metrizable ballean $X\times\bigvee_{\alpha\in A}X_\alpha$.
\end{proof}

By analogy with Theorem~\ref{t9} we can prove

\begin{theorem}\label{t11} The comb $X\comb\limits_{\alpha\in A}X_\alpha$ is normal if the balleans $X$ and $X_\alpha$, $\alpha\in A$, are normal.
\end{theorem}








\section{ Coarse structures, determined by bornologies}

Let $\mathcal{B}$  be a bornology on a set $X$.
We say that a coarse structure $\mathcal{E}$ on $X$ is {\it compatible} with $\mathcal{B}$ if  $\mathcal{B}$ coincides with the bornology $\mathcal B_X$ of all bounded subsets of $(X, \mathcal{E})$.

The family of all coarse structures, compatible with a given bornology $\mathcal B$ has the smallest and largest elements ${\Downarrow}\mathcal B$ and ${\Uparrow}\mathcal B$.

The smallest  coarse structure ${\Downarrow}\mathcal B$ is generated by the base consisting of the entourages $(B\times B)\cup\bigtriangleup_{X}$, where $B\in\mathcal{B}$.

The largest coarse structure ${\Uparrow}\mathcal{B}$
 consists of all entourages  $E\subseteq  X\times X$  such that  $E=E ^{-1}$ and
 $E[B]\in \mathcal{B}$ for each $B\in \mathcal{B}$.

An unbounded ballean $(X,\mathcal E)$ is called
\begin{itemize}
\item {\em discrete} if $\mathcal E={\Downarrow}\mathcal B_X$,
\item {\em antidiscrete} if $\mathcal E={\Uparrow}\mathcal B_X$,
\end{itemize}
where $\mathcal B_X$ is the bornology on $X$, generated by the coarse structure $\mathcal E$.

It can be shown that an unbounded ballean $(X,\mathcal E)$ is discrete if and only if  for every $E\in \mathcal{E}$ there exists a bounded set $B\subset X$ such that
  $E[x]= \{ x\} $  for each $x\in X \backslash B$. In \cite[Chapter 3]{b3} discrete balleans are called pseudodiscrete.
  
An unbounded ballean $(X,\mathcal E)$ is called
\begin{itemize}
\item {\em maximal} if its coarse structure is 
maximal by inclusion in the family of all unbounded coarse structures on $X$;
\item {\em ultradiscrete} if $X$ is discrete and its bornology $\mathcal B_X$ is 
maximal by inclusion in the family of all unbounded bornologies on $X$.
\end{itemize}

 It is clear that each maximal ballean is antidiscrete. 
For maximal balleans, see \cite[Chapter 10]{b3}. For any regular cardinal $\kappa$ the  ballean
  $(\kappa, {\Uparrow} [\kappa]^{<\kappa})$
      is  maximal.
    
It can be shown that each maximal ballean is antidscrete and each ultradiscrete ballean is both discrete and antidiscrete.

\begin{question} Is there a ballean which is discrete and antidiscrete but not ultradiscrete?
\end{question} 

A ballean $(X, \mathcal{E})$  is called {\it ultranormal} if $X$ contains no two unbounded asymptotically disjoint subsets. By  \cite[Theorem  10.2.1]{b3},  every unbounded subset of a  maximal  ballean is large, which implies that each maximal ballean is ultranormal.

\begin{example}\label{t12}For every infinite set $X$, there exists a bornology $\mathcal{B}$  on $X$  such that the  ballean  $(X, {\Uparrow} \mathcal{B})$ is normal but not ultranormal.
\end{example}

\begin{proof} We take two free ultrafilters  $p, q$  on $X$  such that for every mapping  $f: X \to X$,  there exists  $P\in p$  such that $f(P)\notin q$.
We put $\mathcal{B}= \{B\subseteq X: B\notin p, B\notin  q\} $ and note
that  $\mathcal{B}$  is a bornology.

To show that the ballean $(X,{\Uparrow}\mathcal B)$ is normal, it suffices to prove that any disjoint unbounded sets $Y,Z\subset X$ are asymptotically separated.
This will follow as soon as we show that $Y$ is an asymptotic neighborhood of $Y$ and $Z$ is an asymptotic neighborhood of $Z$ in the ballean $(X,{\Uparrow}\mathcal B)$. 

The definition of the bornology $\mathcal B\ni Y,Z$ ensures that one of the sets (say $Y$) belongs to the ultrafilter $p$ and the other (in this case $Z$) belongs to the ultrafilter $q$.

First we show that $Y$ is an asymptotic neighborhood of $Y$ in $(X,{\Uparrow}\mathcal B)$. Given any entourage $E \in {\Uparrow}  \mathcal{B}$, we should check that the set $E[Y]\setminus Y$ is bounded. Let $Y_0=\{y\in Y:E[y]\subset Y\}$ and $Y_1:=Y\setminus Y_0$. Since $Y=Y_0\cup Y_1\in p$, either $Y_0\in p$ or $Y_1\in p$.

If $Y_0\in p$, then $Y_{1}\in \mathcal{B} $  and  $E[Y]\setminus Y\subset E[Y_1]$ is bounded in $(X,{\Uparrow}\mathcal B)$.

If $Y_1\in p$, then for each $y\in Y_{1}$ we can choose an element
  $f(y)\in E[y]\setminus Y$. By the choice of the ultrafilters $p,q$, there exists a set $P\subseteq Y_{1}$, $P\in p$  such  that  $f(P)\notin q$.
Hence, $f(P)\in \mathcal{B}$ and the set $P\subset E^{-1} [f(P)]$ is bounded, which contradicts $P\in p$. This contradiction shows that the case $Y_1\in p$ is not possible and $E[Y]\setminus Y$ is bounded.

By analogy we can prove that $E[Z]\setminus Z$ is bounded. Therefore, the disjoint sets $Y,Z$ have disjoint asymptotically separated neighborhoods and the ballean $X$ is normal. 

Since $X$ can be written as the union $X=Y\cup Z$ of two disjoint (unbounded) sets $Y\in p$, $Z\in q$, the ballean $(X,{\Uparrow}\mathcal B)$ is not ultranormal.
\end{proof}

By a {\em bornological space} we understand a pair $(X,\mathcal B_X)$ consisting of a set $X$ and a bornology $\mathcal B_X$ on $X$. A bornological space $(X,\mathcal B_X)$ is {\em unbounded} if $X\notin\mathcal B_X$. For two bornological spaces $(X,\mathcal B_X)$ and $(Y,\mathcal B_Y)$ their product is the bornological space $(X\times Y,\mathcal B)$ endowed with the bornology $$\mathcal B_{X\times Y}=\{B\subset X\times Y:B\subset B_X\times B_Y\mbox{ for some } B_X\in\mathcal B_X,\;B_Y\in\mathcal B_Y\}.$$

The following theorem allows us to construct many examples of bornological spaces $(X,\mathcal B)$ for which the coarse space $(X,{\Uparrow}\mathcal B)$ is not normal.

\begin{theorem} Let $(X\times Y,\mathcal B)$ be the product of two unbounded bornological spaces $(X,\mathcal B_X)$ and $(Y,\mathcal B_Y)$.
If $\cov(\mathcal B_Y)<\add(\mathcal B_X)$, then the coarse space $(X\times Y,{\Uparrow}\mathcal B)$ is not normal.
\end{theorem}

\begin{proof} Fix any point $(x_0,y_0)\in X\times Y$. Assuming that $\cov(\mathcal B_Y)<\add(\mathcal B_X)$, we shall prove that for a coarse structure $\mathcal E$ on $X\times Y$ is not normal if $\mathcal E$ has the following three properties:
\begin{enumerate}
\item $\mathcal E$ is compatible with the bornology $\mathcal B$;
\item for any $B_Y\in\mathcal B_Y$ there exists $E\in\mathcal E$ such that $X\times B_Y\subset E[X\times\{y_0\}]$;
\item for any $B_X\in\mathcal B$ there exists $E\in\mathcal E$ such that $B_X\times Y\subset E[\{x_0\}\times Y]$.
\end{enumerate}
It is easy to see that the coarse structure ${\Uparrow}\mathcal B$ has these three properties.

By the definition of the cardinal $\kappa=\cov(\mathcal B_Y)$, there there is a family $\{Y_\alpha\}_{\alpha\in\kappa}\subset\mathcal B_Y$ such that $\bigcup_{\alpha\in\kappa}Y_\alpha=Y$.

Assume that $\mathcal E$ is a coarse structure on $X\times Y$ satisfying the conditions (1)--(3). First we check that the sets $X\times\{y_0\}$ and $\{x_0\}\times Y$ are asymptotically disjoint in $(X\times Y,\mathcal E)$. Given any entourage $E\in\mathcal E$, we should prove that the intersection $E[X\times\{y_0\}]\cap E[\{x_0\}\times Y]$ is bounded.  By the condition (1), for every $\alpha\in\kappa$ the bounded set $E^{-1}[E[\{x_0\}\times Y_\alpha]]$ is contained in the product $B_\alpha\times Y$ for some bounded set $B_\alpha\in\mathcal B_X$.  Since $\kappa<\add(\mathcal B_X)$, the union $B_{<\kappa}:=\bigcup_{\alpha\in\kappa}B_\alpha$ belongs  to the bornology $\mathcal B_X$. Given any point $(u,v)\in E[X\times\{y_0\}]\cap E[\{x_0\}\times Y]$, find $x\in X$ and $y\in Y$ such that $(u,v)\in E[(x,y_0)]\cap E[(x_0,y)]$. Since $Y=\bigcup_{\alpha\in\kappa}Y_\alpha$, there exists $\alpha\in\kappa$ such that $y\in Y_\alpha$. Then $(x,y_0)\in E^{-1}[E[(x_0,y)]]\subset E^{-1}[E[\{x_0\}\times Y_\alpha]]\subset B_\alpha\times Y\subset B_{<\kappa}\times Y$ and hence
$(u,v)\in E[(x,y_0)]\subset E[B_{<\kappa}\times\{y_0\}]$, which implies that the intersection $$E[X\times\{y_0\}]\cap E[\{x_0\}\times Y]\subset E[B_{<\kappa}\times \{y_0\}]$$is bounded in $(X\times Y,\mathcal E)$.

Assuming that the coarse space $(X\times Y,\mathcal E)$ is normal, we can find disjoint asymptotical neighborhoods $U$ and $V$ of the asymptotically disjoint sets $X\times\{y_0\}$ and $Y\times\{x_0\}$. By the condition (2), for every $\alpha\in\kappa$ there exists an entourage $E_\alpha\in\mathcal E$ such that $X\times Y_\alpha\subset E_\alpha[X\times\{y_0\}]$. Since $U$ is an asymptotic neighborhood of the set $X\times\{y_0\}$ in $(X\times Y,\mathcal E)$, the set $(X\times Y_\alpha)\setminus U\subset E_\alpha[X\times\{y_0\}]\setminus U$ is bounded in $(X\times Y,\mathcal E)$. Now the condition (1) implies that $(X\times Y_\alpha)\setminus U\subset D_\alpha\times Y$ for some bounded set $D_\alpha\in\mathcal B_X$.

 We claim that the family $\{D_\alpha\}_{\alpha\in\kappa}$ is cofinal in $\mathcal B_X$. Indeed, given any bounded set $D\in\mathcal B_X$, use the condition (3) and find a entourage $E\in\mathcal E$ such that $D\times Y\subset E[\{x_0\}\times Y]$. Since $V$ is an asymptotic neighborhood of the set $\{x_0\}\times Y$, the set $E[\{x_0\}\times Y]\setminus V$ is bounded in $(X\times Y,\mathcal E)$ and the condition (1) ensures that it has bounded projection onto $Y$.
Since $Y\notin\mathcal B_Y$, we can find a point $y\in Y$ such that $X\times\{y\}$ is disjont with $E[\{x_0\}\times Y]\setminus V$. Find $\alpha\in\kappa$ with $y\in Y_\alpha$.  Then $(X\times \{y\})\cap E[\{x_0\}\times Y]\subset V$ and hence
$$D\times \{y\}\subset (X\times y\})\cap E[\{x_0\}\times Y]\subset (X\times Y_\alpha)\cap V\subset (X\times Y_\alpha)\setminus U\subset D_\alpha\times Y,$$ which yields the desired inclusion $D\subset D_\alpha$. Therefore,
$$\cof(\mathcal B_X)\le|\{D_\alpha\}_{\alpha\in \kappa}|\le\kappa=\cov(\mathcal B_Y)<\add(\mathcal B_X),$$
which contradicts the known inequality $\add(\mathcal B_X)\le\cof(\mathcal B_X)$.
\end{proof}

\end{document}